\documentclass[11pt]{amsart}
\usepackage{amsmath,amsthm,amsopn,amsfonts,amssymb}
\usepackage{mathrsfs,a4wide}
\usepackage{graphicx,graphics,color}
\newtheorem{thm}{Theorem}[section]

\newtheorem{lemma}[thm]{Lemma}

\theoremstyle{definition}
\newtheorem{defn}[thm]{Definition}
\theoremstyle{remark}
\newtheorem{rem}[thm]{Remark}

\numberwithin{equation}{section}

\newcommand{\norm}[1]{\left\Vert#1\right\Vert}

\newcommand{\Real}{\mathbb R}

\newcommand{\N}{\mathbb{N}}
\newcommand{\R}{\mathbb{R}}

\begin{document}

\title{An isoperimetric inequality for a nonlinear eigenvalue problem}
\author{Gisella Croce, Antoine Henrot and Giovanni Pisante}
\address{Laboratoire de Math\'ematiques Appliqu\'ees du Havre, Universit\'e du Havre, 25, rue Philippe Lebon, 76063 Le Havre (FRANCE)
\\
Institut \'Elie Cartan Nancy, Nancy Universit\'e-CNRS-INRIA, B.P. 70239, 54506 Vandoeuvre les Nancy (FRANCE)
\\
Dipartimento di Matematica, Seconda Universit\`a degli studi di Napoli, Via Vivaldi, 43, 81100 Caserta (ITALY)}
\email{gisella.croce@univ-lehavre.fr, henrot@iecn.u-nancy.fr, giovanni.pisante@unina2.it}

\maketitle

\begin{abstract}
We prove an isoperimetric inequality of the Rayleigh-Faber-Krahn type for
a nonlinear generalization of the first twisted Dirichlet eigenvalue. More
precisely, we show that the minimizer among sets of given volume is the
union of two equal balls.
\end{abstract}

\section{Introduction}
In this note we study a  generalized version of the so called twisted Dirichlet eigenvalue problem. More precisely, for $\Omega$ an open bounded subset of $\R^N$ we set
\begin{equation}\label{wirtinger}
\lambda^{p,q}(\Omega)=
\inf
\left\{
\frac{\norm{\nabla v}_{L^p(\Omega)}}{\norm{v}_{L^q(\Omega)}}, v\neq 0, v\in W^{1,p}_{0}(\Omega), \int\limits_{\Omega}|v|^{q-2}v\,dx=0\,
\right\}.
\end{equation}
Among the sets $\Omega$ with fixed volume, we are interested in characterizing those which minimize  $\lambda^{p,q}(\Omega)$. In other words we look for an isoperimetric inequality of Rayleigh-Faber-Krahn type. This kind of inequality is related to the optimization of the first eigenvalue for the Dirichlet problem associated to nonlinear operators in divergence form and have been widely studied for functionals that do not involve mean constraints. In such cases a rearrangement technique proves that the minimizing set is a ball and several results concerning its stability are also available (see for instance \cite{Nazarov:2001p297},\cite{1158.35069},\cite{FusMagPra06-000}). When  mean type constraints are considered together with the Dirichlet boundary condition in an eigenvalue problem, the optimization problem becomes more difficult, since one is lead to deal with non local problems. Due to the fact that an eigenfunction for $\lambda^{p,q}(\Omega)$  is forced to change sign inside $\Omega$, and hence has at least  two nodal domains,  one cannot expect in general
to have a radial optimizer.

The adjective \emph{twisted} was introduced by Barbosa and B\'erard in \cite{Barbosa:2000p18},
in the study of spectral properties of the second variation of a constant mean curvature immersion of a Riemannian manifold.
In that framework a Dirichlet eigenvalue problem arose naturally  with a vanishing mean constraint. The condition on the mean value comes from the fact that the variations under consideration preserve some balance of volume.

Further results in this direction can be found in the paper of Freitas and Henrot \cite{FreHen04-000}, where, dealing with the linear case, the authors solved
the shape optimization problem for the first twisted Dirichlet eigenvalue. In particular they considered $\lambda^{2,2}(\Omega)$, and they proved that the only optimal shape is given by a pair of balls of equal measure. The one-dimensional case has also attracted much interest.
In \cite{DacGanSub92-000}, Dacorogna, Gangbo and Sub\'ia studied the following generalization of the Wirtinger inequality
\begin{equation}\label{wir}
\inf\left\{  \frac{\|u'\|_{L^{p}((-1,1))}}{\|u\|_{L^{q}((-1,1))}}\;,\;u\in W^{1,p}(-1,1)\setminus\{0\}\;,\;u(-1)=u(1)=0\;,\;\int_{-1}^{1}|u|^{q-2}u\,dx=0\right\}
\end{equation}
for $p,q>1$
proving that the optimizer is an odd function. Moreover they explained the connection between the value of $\lambda^{p,p'}((-1,1))$, where $p=\frac{p}{p-1}$, and an isoperimetric inequality. Indeed, let $A\subset \R^2$ whose boundary is a simple closed curve $t \in [-1,1]\to (x(t),y(t))$ with $x, y \in W^{1,p}_0((-1,1))$.
Let
$$
L(\partial A)=\int_{-1}^{1}(|x'(t)|^p + |y'(t)|^p)^{\frac 1p}dt
$$
and
$$
M(A)=\frac 12\int_{-1}^{1}[y'(t)x(t)-y(t)x'(t)]dt\,.
$$
Then $L^2(\partial A)-4\lambda^{p,p'}((-1,1))M(A)\geq 0$. The case of equality holds if and only if $A=\{(x,y)\in \R^2: |x|^{p'}+|y|^{p'}=1\}$,
up to a translation and a dilation.

Several other results are available in the one-dimensional case, see for instance \cite{MR1974431}, \cite{busl}, \cite{MR1701790}, \cite{MR1757396}, \cite{Farroni:2010p643} and the references therein for further details.

Our aim here, as in \cite{FreHen04-000}, is to prove that the optimal shape for $\lambda^{p,q}(\Omega)$
is a pair of equal balls. The main result can be stated as follows.

\begin{thm}\label{main_theorem}
Let $\Omega$ be an open bounded subset of $\R^N$.
Then, for
\begin{equation}\label{hypotheses_pqN}
1<p<\infty\,\, \textnormal{and}\,\,\,
\left\{
\begin{array}{l}
1<q< p^* \,,\,\,\textnormal{if}\,\, 1<p<N
\\
1<q<\infty\,,\,\,\textnormal{if}\,\, p\geq N
\end{array}
\right.
\end{equation}
we have
\[
\lambda^{p,q}(\Omega)\geq \lambda^{p,q}(B_{1}\cup B_{2}),
\] where $B_{1}$ and $B_{2}$ are disjoint  balls of measure $|\Omega|/2$.
\end{thm}
The rest of the paper is devoted to the proof of Theorem \ref{main_theorem} and it is divided into  two steps. In the first one,
using the symmetrization technique,
 we  show that it is enough to minimize the functional $\lambda^{p,q}$ on sets given by the union of two disjoint balls $B_1$ and $B_2$ (not necessarily equal) and to identify the minimizing pairs.
Moreover we write the Euler equation for a minimizer $u$ of $\lambda^{p,q}(B_1\cup B_2)$, proving that the Lagrange multiplier associated to the constraint $$\displaystyle \int_{\Omega}|u|^{q-2}u=0$$ is zero (cf. Theorem \ref{idee_daco_gangbo_subia}).

The second step, which consists in showing  that the two optimal balls have to be equal, is more subtle. In the case $p=q=2$ solved in \cite{FreHen04-000}, the proof is based
on the explicit formula for the (radial) solutions to the Euler equation of the functional and
on fine properties of the zeroes of Bessel functions. Here we use a more geometric argument obtaining as a byproduct a simpler proof of the results of Freitas and Henrot.
More precisely  $\lambda^{p,q}(B_1 \cup B_2)$ is attained at a function
$u=u_1\chi_{B_1}-u_2\chi_{B_2}$, with $u_{1}$ and $u_{2}$  radial positive functions.
If we look at $\lambda^{p,q}(B_1 \cup B_2)$
as a function of sets we obtain the following optimality condition  from  the domain derivative (cf. Theorem \ref{Antoine}):
\[
\left|\frac{\partial u_1}{\partial \nu_1}\right|=\left|\frac{\partial u_2}{\partial \nu_2}\right|\,.
\]
From the other hand, the divergence theorem applied to the Euler equation gives that
\[
 \left|\frac{\partial u_1}{\partial \nu_1}\right|^{p-1}|\partial B_1|=
\left|\frac{\partial u_2}{\partial \nu_2}\right|^{p-1}|\partial B_2|\,.
\]
This,  combined with the previous condition, implies that $B_{1}$ and $B_{2}$ have the same measure.

\section{The first generalized twisted eigenvalue}

We start our study proving that the the value $\lambda^{p,q}(\Omega)$ is  attained for any choice of a bounded open set $\Omega\subset \R^N$.

\begin{lemma}
Assume (\ref{hypotheses_pqN}). Then $\lambda^{p,q}(\Omega)>0$ and there exists a bounded function $u\in W^{1,p}_0(\Omega)$
such that
$$
\lambda^{p,q}(\Omega)=\frac{\norm{\nabla u}_{L^p(\Omega)}}{\norm{u}_{L^q(\Omega)}} \,\,\textrm{and}\,\,\,\int_{\Omega}|u|^{q-2}u\,dx=0\,.
$$
\end{lemma}
\begin{proof}
Let
$$\displaystyle H_{\frac1n}(v)=
\norm{\nabla v}^p_{L^p(\Omega)}-\left([\lambda^{p,q}(\Omega)]^p+\frac 1n\right)\norm{v}^p_{L^q(\Omega)}$$
with $n\in \N$, and
$$
\displaystyle G(v)=\int_{\Omega} |v|^{q-2}v\,dx.
$$
By definition of infimum, for every $n$ there exists $u_n$
such that
$$
\displaystyle \int_{\Omega}|u_n|^{q-2}u_n\,dx=0\,,\,\,\,\,\,\,\,\,\,\,\,\,H_{\frac 1n}(u_n)<0\,.
$$ Without loss of generality we can assume that $\norm{\nabla u_n}_{L^p(\Omega)}=1.$
By Poincar\'e inequality,
$\norm{u_n}_{W^{1,p}(\Omega)}$ is uniformly bounded. Since $p>1$, up to a subsequence,
$u_n$ converges weakly to some  $u \in W^{1,p}_0(\Omega)$. By hypotheses
 (\ref{hypotheses_pqN}) on $p$ and $q$,
$u_n\to v$ in $L^q(\Omega)$ and then
$$
\displaystyle \int_{\Omega}|u|^{q-2}u\,dx=0.
$$
This implies that
$$
\norm{\nabla u}^p_{L^p(\Omega)}-[\lambda^{p,q}(\Omega)]^p\norm{u}^p_{L^q(\Omega)}\leq 0.
$$
By definition of $\lambda^{p,q}(\Omega)$, necessarily we have
$$
\norm{\nabla u}^p_{L^p(\Omega)}-[\lambda^{p,q}(\Omega)]^p\norm{u}^p_{L^q(\Omega)}=0.
$$
To prove that $u \neq 0$ and $\lambda^{p,q}(\Omega)>0$, it is sufficient to pass to the limit in
$H_n(u_n)<0$ to get
$$
1\leq [\lambda^{p,q}(\Omega)]^p\norm{u}^p_{L^q(\Omega)}.
$$
We are now going to prove that $u$ is bounded.
For $\varphi, \theta \in C^{\infty}_0(\Omega)$, let
$$
\Phi(\varepsilon, t)=
\int\limits_{\Omega}|\nabla u\,dx + \varepsilon \nabla \varphi+t\nabla \theta|^p\,dx
-[\lambda^{p,q}(\Omega)]^p
\left[\int\limits_{\Omega}|u + \varepsilon \varphi+t\theta|^q\,dx
\right]^{\frac pq}\,.
$$
Let $\theta$ be such that $$
\displaystyle (q-1)\int_{\Omega}|u|^{q-2}\theta\,dx=1,$$
if such $\theta$ does not exist one would have $|u|^{q-2}=0$ which is a contradiction. Then, set
$$
\psi(\varepsilon, t)=
\int\limits_{\Omega}|u + \varepsilon \varphi+t\theta|^{q-2}(u + \varepsilon \varphi+t\theta)\,dx\,.
$$
The hypotheses on $\theta$ imply that $\psi_t(0,0)=1$.
By the implicit function theorem applied to $\psi$, there exists a function $\tau$ such that
$\psi(\varepsilon, \tau(\varepsilon))$=0 and $\tau'(0)=-\psi_{\varepsilon}(0,0)$. Since $(0,0)$ is a minimizer for $\Phi$, we deduce that
\begin{equation}\label{implicit_function_thm_primo}
\Phi_{\varepsilon}(0,0)+\Phi_t(0,0)\tau'(0)=\Phi_{\varepsilon}(0,0)-\Phi_t(0,0)\psi_{\varepsilon}(0,0)=0\,.
\end{equation}

By explicit calculations we have
$$
\Phi_{\varepsilon}(0,0)=
p\int\limits_{\Omega}|\nabla u|^{p-2}\nabla u\cdot \nabla \varphi\,dx
-[\lambda^{p,q}(\Omega)]^p p
\norm{u}_{L^q(\Omega)}^{p-q}
\int\limits_{\Omega}|u|^{q-2}u\varphi\,dx;
$$
$$
\Phi_{t}(0,0)=
p\int\limits_{\Omega}|\nabla u|^{p-2}\nabla u\cdot \nabla \theta\,dx
-[\lambda^{p,q}(\Omega)]^pp
\norm{u}_{L^q(\Omega)}^{p-q}
\int\limits_{\Omega}|u|^{q-2}u\theta\,dx;
$$
$$
\psi_{\varepsilon}(0,0)=
(q-1)\int\limits_{\Omega}|u|^{q-2} \varphi\,dx\,,
$$
the equation (\ref{implicit_function_thm_primo})
is indeed equivalent to
$$
\int\limits_{\Omega}|\nabla u|^{p-2}\nabla u\cdot \nabla \varphi\,dx
-[\lambda^{p,q}(\Omega)]^p \norm{u}_{L^q(\Omega)}^{p-q}
\int\limits_{\Omega}|u|^{q-2}u\varphi\,dx
=
\mu_0
(q-1)\int\limits_{\Omega}|u|^{q-2} \varphi\,dx
$$
with
$$
\mu_0=\int\limits_{\Omega}|\nabla u|^{p-2}\nabla u\cdot \nabla \theta\,dx
-[\lambda^{p,q}(\Omega)]^p
\norm{u}_{L^q(\Omega)}^{p-q}
\int\limits_{\Omega}|u|^{q-2}u\theta\,dx
\,.
$$
By standard regularity results on elliptic equations (see $\S$ 5 of Chapter 2 in \cite{Lad-Ural}) we deduce that $u$ is bounded.
\end{proof}

Let now $u\in W^{1,p}_0(\Omega)$ be such that
$$
\displaystyle\lambda^{p,q}(\Omega)=\frac{\norm{\nabla u}_{L^p(\Omega)}}{\norm{u}_{L^q(\Omega)}}\,\;\;\textrm{and}\;\;\; \int_{\Omega}|u|^{q-2}u\,dx=0
$$
and set $$\Omega_{+}=\{x \in \Omega: u(x)>0\}\,,\quad \Omega_{-}=\{x \in \Omega: u(x)<0\}.$$

Our aim is to prove that we can reduce to the case of two balls. We will use a technique used in \cite{FreHen04-000} based on the Schwarz rearrangement. Here we recall just the definition and the properties that we will need in the proof. For more details on rearrangement techniques we refer to \cite{Hen06-000} and \cite{Kawohl:1985p448}.

\begin{defn}
For a measurable set $\omega\subset \R^N$, we denote by $\omega^*$ the ball of same measure as $\omega$.
If $u$ is a non-negative measurable function defined on a measurable set $\Omega$ and $u=0$ on $\partial \Omega$, let
$$\Omega(c)=\{x \in \Omega: u(x)\geq c\}.$$ The Schwarz rearrangement of $u$ is the function $u^*$ defined on $\Omega^*$ by
$$u^*=\sup \{c: x \in \Omega(c)^*\}.$$
\end{defn}
The next theorem summarizes some of the main properties of the Schwarz symmetrization.
\begin{thm}\label{symm}
Let $u$ be a non-negative measurable function defined on a measurable set $\Omega$ with $u=0$ on $\partial \Omega$. Then
\begin{enumerate}
\item
$u^*$ is a radially symmetric non-increasing function of $|x|$;
\item for any measurable function $\psi: \R^+ \to \R$
$$\displaystyle \int_{\Omega}\psi(u)\,dx=\int_{\Omega^*}\psi(u^*)\,dx \;\;;$$ \item
if $u \in W^{1,p}_0(\Omega)$, then $u^* \in W^{1,p}_0(\Omega^*)$
and
$$
\displaystyle \int_{\Omega}|\nabla u|^p\,dx\geq \int_{\Omega^*}|\nabla u^*|^p\,dx.
$$
\end{enumerate}
\end{thm}
Using the Schwarz symmetrization and suitable constrained variations we are now able to reduce our problem to the ``radial'' one. Indeed we have the following theorem.
\begin{thm}\label{thmpremiereinegalite}
Let $B_{\pm}$ be a ball of same measure as
$|\Omega_{\pm}|$. Then
$\lambda^{p,q}(\Omega)\geq \lambda^{p,q}(B_+\cup B_-)$.
\end{thm}
\begin{proof}
Let $u_+=u\lfloor_{\Omega_+}$ and  $u_-=-u\lfloor_{\Omega_-}$.
By symmetrizing $u_{+}$ and $u_{-}$ respectivelly, by the properties of Schwarz rearrangement (cfr. Theorem \ref{symm}) we can write
$$
[\lambda^{p,q}(\Omega)]^p
\geq
\frac{\norm{\nabla u_+^*}^p_{L^p(B_+)}+\norm{\nabla u_-^*}^p_{L^p(B_-)}}{\left[\norm{u_+^*}^q_{L^q(B_+)}+\norm{u_-^*}^q_{L^q(B_-)}\right]^{\frac pq}}\,.
$$
Moreover, by equimeasurability ensured by Theorem \ref{symm}.(2), using the volume constraint, we deduce that
$$
0=\int\limits_{\Omega_+}|u_+|^{q-2}u_+\,dx-\int\limits_{\Omega_-}|u_-|^{q-2}u_-\,dx
=\int\limits_{B_+}|u_+^*|^{q-2}u^*_+\,dx-\int\limits_{B_-}|u^*_-|^{q-2}u^*_-\,dx\,.
$$
If we set
$$
\lambda^{*}
=
\inf_A\frac{\norm{\nabla f}^p_{L^p(B_+)}+\norm{\nabla g}^p_{L^p(B_-)}}{\left[\norm{f}^q_{L^q(B_+)}+\norm{g}^q_{L^q(B_-)}\right]^{\frac pq}}
$$
where $A$ is defined by
$$
A=\left\{
(f,g)\in W^{1,p}_0(B_+)\times W^{1,p}_0(B_-): \int_{B_+}|f|^{q-2}f=\int_{B_-}|g|^{q-2}g
\right\}\,,
$$
we clearly have
\begin{equation}\label{prima_stima}
[\lambda^{p,q}(\Omega)]^p\geq \lambda^*\,.
\end{equation}
It is easily seen that $\lambda^*$ is attained in $(f_+, f_-)$, with $f_+, f_-\geq 0$. Without loss of generality
we can moreover assume that
\begin{equation}\label{cons}
\displaystyle \int\limits_{B_+}|f_+|^{q}\,dx
+
\int\limits_{B_-}|f_-|^{q}\,dx=1.
\end{equation}
For $\varphi_{\pm}, \theta_{\pm} \in C^{\infty}_0(B_{\pm})$, define

\begin{equation*}
\begin{split}
\Phi(\varepsilon, t)= & \int\limits_{B_+}|\nabla f_+\,dx + \varepsilon \nabla \varphi_++t\nabla \theta_+|^p\,dx
+
\int\limits_{B_-}|\nabla f_-\,dx + \varepsilon \nabla \varphi_-+t\nabla \theta_-|^p\,dx +
\\
 & -\lambda^*\left[
\int\limits_{B_+}|f_+ + \varepsilon \varphi_++t\theta_+|^q\,dx
+
\int\limits_{B_-}|f_- + \varepsilon \varphi_-+t\theta_-|^q\,dx\right]^{\frac pq}\,.
\end{split}
\end{equation*}
Let $(\theta_+,\theta_-)$ be such that
$$
\displaystyle (q-1)\int_{B_+}|f_+|^{q-2}\theta_+\,dx-(q-1)\int_{B_-}|f_-|^{q-2}\theta_-\,dx=1\,.
$$
Such choice of $(\theta_+,\theta_-)$ is possible, since, if not, one would have $|f_+|^{q-2}=|f_-|^{q-2}=0$, that contradicts \eqref{cons}.
If we define the functional
\begin{equation*}
\begin{split}
\psi(\varepsilon, t)= &
\int\limits_{B_+}|f_+ + \varepsilon \varphi_++t\theta_+|^{q-2}(f_+ + \varepsilon \varphi_++t\theta_+)\,dx + \\
& - \int\limits_{B_-}|f_- + \varepsilon \varphi_-+t\theta_-|^{q-2}(f_- + \varepsilon \varphi_-+t\theta_-)\,dx\,,
\end{split}
\end{equation*}
the hypotheses on $(\theta_+,\theta_-)$ imply that $\psi_t(0,0)=1$.
By the implicit function theorem applied to $\psi$, there exists a function $\tau$ such that
$\psi(\varepsilon, \tau(\varepsilon))$=0 and $\tau'(0)=-\psi_{\varepsilon}(0,0)$. Since $(0,0)$ is a minimizer for $\Phi$,
$$
\Phi_{\varepsilon}(0,0)+\Phi_t(0,0)\tau'(0)=\Phi_{\varepsilon}(0,0)-\Phi_t(0,0)\psi_{\varepsilon}(0,0)=0\,,
$$
that is,
\begin{equation*}
\begin{split}
\mu_0 (q-1)&\left[\int\limits_{B_+}|f_+|^{q-2} \varphi_+\,dx
-\int\limits_{B_-}|f_-|^{q-2} \varphi_-\,dx\right] =
\int\limits_{B_+}|\nabla f_+|^{p-2}[\nabla f_+ ]\cdot \nabla \varphi_+\,dx \, +
\\
 & + \int\limits_{B_-}|\nabla f_-|^{p-2}\nabla f_-\cdot \nabla \varphi_-\,dx -\lambda^* \left[
\int\limits_{B_+}|f_+|^{q-2}f_+\varphi_+\,dx +
\int\limits_{B_-}|f_-|^{q-2}f_-\varphi_-\,dx
\right]
\end{split}
\end{equation*}
with
\begin{equation*}
\begin{split}
\mu_0 = & \int\limits_{B_+}|\nabla f_+|^{p-2}\nabla f_+\cdot \nabla \theta_+\,dx
+
\int\limits_{B_-}|\nabla f_-|^{p-2}\nabla f_-\cdot \nabla \theta_-\,dx \,+\\
& -\lambda^*
\left[
\int\limits_{B_+}|f_+|^{q-2}f_+\theta_+\,dx
+
\int\limits_{B_-}|f_-|^{q-2}f_-\theta_-\,dx
\right]\,.
\end{split}
\end{equation*}
It follows that
$
w=f_+\chi_{B_+}-f_-\chi_{B_-}
$
satisfies on $B_+ \cup B_-$ the equation
\begin{equation}\label{premiere_equation_euler}
-{\rm div}(|\nabla w|^{p-2}\nabla w)=\lambda^*|w|^{q-2}w+\mu_0(q-1)|w|^{q-2}\,.
\end{equation}
Now we observe that multiplying \eqref{premiere_equation_euler} by $w$, one has
\begin{equation*}
\begin{split}
\frac{\displaystyle\int\limits_{B_+\cup B_-}|\nabla w|^p\,dx}{\displaystyle\int\limits_{B_+\cup B_-}|w|^q\,dx} & =\lambda^*\geq
\inf
\left\{
{\int\limits_{B_+\cup B_-}|\nabla v|^p\,dx}\,, v: \int\limits_{B_+\cup B_-}|v|^{q-2}v\,dx=0, \norm{v}_{L^q(B_+\cup B_-)}=1
\right\}
\\
& = [\lambda^{p,q}(B_+\cup B_-)]^p\,.
\end{split}
\end{equation*}
The above inequality and (\ref{prima_stima}) imply that $\lambda^{p,q}(\Omega)\geq \lambda^{p,q}(B_+\cup B_-)$.
\end{proof}

We are now going to write the Euler equation for $\lambda^{p,q}(\Omega)$ in the case where $\Omega$ is the union of two disjoint balls. We will make use of a technique introduced in \cite{DacGanSub92-000} to carefully choose the variations.


\begin{thm}\label{idee_daco_gangbo_subia}
Let $\Omega=B_1\cup B_2$ where $B_1$ and $B_2$ are two disjoint balls.
Let $u\in W^{1,p}_{0}(\Omega)$  be a bounded function such that $\displaystyle \int\limits_{\Omega}|u|^{q-2}u\,dx=0$
and
$
\displaystyle \lambda^{p,q}(\Omega)=\frac{\norm{\nabla u}_{L^p(\Omega)}}{\norm{u}_{L^q(\Omega)}}\,.
$
Then
\begin{equation}\label{equazione_eulero_dgs}
-{\rm div}(|\nabla u|^{p-2}\nabla u)=[\lambda^{p,q}(\Omega)]^p\norm{u}_{L^q(\Omega)}^{p-q}|u|^{q-2}u\,.
\end{equation}
\end{thm}
\begin{proof}
We set
$$
G(v)=\int_{\Omega} |v|^{q-2}v\,dx
\,,\,\,\,\,\,\,\,\,\,\,F(v)=\int_{\Omega} |\nabla v|^{p}\,dx-[\lambda^{p,q}(\Omega)]^p\left[\int_{\Omega} |v|^{q}\,dx\right]^{p/q}\,.
$$
Let $\varphi \in C^{\infty}_0(\Omega)$ and $t \in (0,1)$.
Let
$$
\begin{array}{cccc}
\Psi:& \R&\to& \R
\\
& \beta &\to & G(u+t\varphi +\beta)\,.
\end{array}
$$
Then $\Psi$ is continuous,  $\Psi(-1-\norm{u+t\varphi}_{L^{\infty}(\Omega)})<0$ and
$\Psi(1+\norm{u+t\varphi}_{L^{\infty}(\Omega)})>0$.
By continuity there exists $\beta_t \in \R$ such that
$\Psi(\beta_t)=G(u+t\varphi +\beta_t)=0$.

Let $t\in (0,1)$ fixed.
We set $\displaystyle c_t=\frac{\beta_t}{t}$. We are going to prove the existence of a sequence  $t_n\to 0$ such that
 $c_{t_n}$ has a finite limit as $n\to \infty$ (up to a subsequence).
If there exists a sequence $t_n\to 0$ and $x_{t_n} \in \Omega$ such that
$\varphi(x_{t_n})+c_{t_n}=0$, then we have the result, since $\varphi$ is bounded.
If there exists $\delta>0$ such that, for every $0<t< \delta$,
 $\varphi(x)+c_t\neq 0$ for every $x \in \Omega$, let us show that
 $\varphi(x)+c_t$ must change sign in $\Omega$. Otherwise, by the strict convexity of $s\to |s|^q$ (and then by the strict monotonicity of $s\to |s|^{q-2}s$) we should have
$$
\int_{{\Omega}}|u+t\varphi +\beta_t|^{q-2}(u+t\varphi +\beta_t)\,dx>
\int_{{\Omega}}|u|^{q-2}u\,dx
$$
or
$$
\int_{{\Omega}}|u+t\varphi +\beta_t|^{q-2}(u+t\varphi +\beta_t)\,dx<
\int_{{\Omega}}|u|^{q-2}u\,dx\,,
$$
that is,
$$
0=\int_{\Omega}|u+t\varphi +\beta_t|^{q-2}(u+t\varphi +\beta_t)\,dx>\int_{{\Omega}}|u|^{q-2}u\,dx=0
$$
or
$$
0=\int_{\Omega}|u+t\varphi +\beta_t|^{q-2}(u+t\varphi +\beta_t)\,dx<\int_{{\Omega}}|u|^{q-2}u\,dx=0
$$
which is a contradiction.

Then, for $0<t<\delta$,
on a subset of $\Omega$ one has $\varphi(x)+c_t>0$ and on its complement
$\varphi(x)+c_t< 0$. This implies that $|c_t|\leq \norm{\varphi}_{L^{\infty}(\Omega)}$.
Therefore there exists $0<t_n <\delta$ such that $t_n \to 0$ and
$c_{t_n} \to c$ as $n\to +\infty$.

We have
$$
<F'(u),\varphi>=p\int_{\Omega}|\nabla u|^{p-2}\nabla u \cdot \nabla \varphi\,dx-
p[\lambda^{p,q}(\Omega)]^p \norm{u}_{L^q(\Omega)}^{p-q}\int_{\Omega}|u|^{q-2}u\varphi\,dx\,.
$$
On the other hand,
$$
0\leq \lim_{n\to \infty}\frac{F(u+t_n(\varphi+c_{t_n}))-F(u)}{t_n}=
$$
$$
=<F'(u),\varphi>-c\,p\,[\lambda^{p,q}(\Omega)]^p \norm{u}_{L^q(\Omega)}^{p-q}
\int_{\Omega}|u|^{q-2}u\,dx
=<F'(u),\varphi>\,.
$$
The previous inequality implies that
$$
\int_{\Omega}|\nabla u|^{p-2}\nabla u \cdot \nabla \varphi\,dx=
[\lambda^{p,q}(\Omega)]^p \norm{u}_{L^q(\Omega)}^{p-q}\int_{\Omega}|u|^{q-2}u\varphi\,dx
$$
for every $\varphi \in C^{\infty}_0(\Omega).$

\end{proof}

\section{The shape optimization problem}

In this section we are going to find a geometrical necessary condition for a set $\Omega$ to be a minimizer of $\lambda^{p,q}(\Omega)$, where $\Omega$ is the union of two disjoint balls. We will exploit the derivative with respect to the domain of the set functional $\lambda^{p,q}(\Omega)$ and investigate an optimality condition, i.e. we will identify the domains with vanishing domain derivative. Here we briefly recall, for the reader's convenience, the ideas underlying the concept of domain derivative and we refer for instance to \cite{HenPie05-000} and  \cite{Simon:1980p300} for a detailed description of the theory and for further details on its applicability.

Roughly speaking the domain derivative can be understood in the following way. Let  $\Omega$ be a bounded smooth domain in $\Real^{n}$, $V:\Real^{n} \to \Real^{n}$ be a sufficiently smooth vector field, $t\geq 0$ and denote by $\Omega_{t}$ the image of $\Omega$ under the map $I+tV$, where $I$ stands for the identity. Let us consider the boundary value problem
\begin{equation}\label{BVP}
\left\{
\begin{array}{cl}
\mathcal{A}(t,u)=0 & \;\;\textrm{in }\Omega_{t} \\
u=0 &\;\; \textrm{on }\partial\Omega_{t}
\end{array}
\right.
\end{equation}
and an integral functional given by
\[
\mathcal{J}(t):=\int_{\Omega_{t}} C(u)\,dx,
\]
with $\mathcal{A}$ and $C$ are differential operators acting on a space of functions defined in $\Omega_{t}$. Under suitable regularity hypotheses, the function $t\to u_{t}$, that associates to $t$ the solution of  problem \eqref{BVP}, is differentiable and its derivative in zero, denoted by $\dot{u}:=u_{t}'(0)$, satisfies the following conditions
\begin{equation}\label{BVP-limit}
\left\{
\begin{array}{cl}
\partial_{t}\mathcal{A}(0,u_{0})+\partial_{t}\mathcal{A}(0,u_{0}) \dot{u}=0 & \textrm{in }\Omega
\\
\displaystyle \dot{u}=-\frac{\partial u_0}{\partial \nu} V\cdot \nu & \textrm{on }\partial\Omega
\end{array}
\right.
\end{equation}
where $\nu$ is the outward unit normal to $\partial \Omega$. Moreover, we can calculate the domain derivative for $t=0$ of the functional $\mathcal{J}$ in the direction $V$ as
\begin{equation}\label{Functional-Derived}
\mathcal{J}'(0)=\int_{\Omega} \partial_{u} C \, \dot{u} \,dx+\int_{\partial \Omega} C(u_{0}) \,V\cdot \nu\, d\mathcal{H}^{N-1}.
\end{equation}

The results of the previous section ensure us that we can restrict our study to the sets $\Omega=B_{1}\cup B_{2}$ where $B_1$ and $B_2$ are two disjoint balls of radius $R_1$ and $R_2$ respectively such that $|B_1\cup B_2|=\omega_N$, where $\omega_N$ is the measure of the unit ball in $\R^N$.
Let $u$ be the minimizer function realizing the value $\lambda^{p,q}(\Omega)$.
Using the Schwarz rearrangement as in Theorem \ref{thmpremiereinegalite}  we can assume that  $\lambda^{p,q}(B_1\cup B_2)$ is attained at a function
$u=u_1\chi_{B_1}-u_2\chi_{B_2}$,  with $u_{1}$ and $u_{2}$ non negative radial functions on $B_{1}$ and $B_{2}$ respectively.

For this kind of domains, by Theorem \ref{idee_daco_gangbo_subia}, $u$ satisfies  \eqref{equazione_eulero_dgs}. By scaling invariance, it is not restrictive to deal with solutions that satisfy the condition
\begin{equation}\label{vol-constr}
\int_{\Omega} |u|^{q}\,dx=1\,.
\end{equation}
Thus we are lead to consider $u$ satisfying \eqref{vol-constr}, the constraint
\begin{equation*}
\int_{\Omega} |u|^{q-2}v\,dx=0
\end{equation*}
and the Dirichlet eigenvalue problem
\begin{equation*}
\left\{
\begin{array}{cl}
-{\rm div}(|\nabla u|^{p-2}\nabla u)=[\lambda^{p,q}(\Omega)]^p |u|^{q-2}u & \;\;\textrm{in }\Omega \\
u=0 &\;\; \textrm{on }\partial\Omega\,.
\end{array}
\right.
\end{equation*}

Observe that in dimension 1, the minimizer of $\lambda^{p,q}((a,b))$ is an anti-symmetric function with respect to $(\frac{a+b}{2},0)$, as proved in \cite{DacGanSub92-000}.
Therefore in the sequel we will assume that $N\ge 2$.

Clearly an optimal set, i.e. a set that minimizes $\lambda^{p,q}(\Omega)$, will be a critical set with respect to the domain variations.
If we prove that the eigenvalue has a domain derivative $\dot{\lambda}^{p,q}(\Omega)$, which will be true if
 $\lambda^{p,q}(\Omega)$ is a simple eigenvalue, then $\dot{\lambda}^{p,q}(\Omega)=0$ (see for example \cite{HenPie05-000} for further details and proof of the differentiability of a simple eigenvalue).
 This motivates the next theorem.






\begin{thm}\label{simplicity}
Let $\Omega=B_1\cup B_2$, where $B_1$ and $B_2$ are two disjoint balls. Then $\lambda^{p,q}(\Omega)$ is a simple eigenvalue, i.e. there exists a unique function $u=u_1\chi_{B_1}-u_2\chi_{B_2}$, modulo a multiplicative constant, that realizes
 $$
 \lambda^{p,q}(\Omega)=\frac{\norm{\nabla u}_{L^p(\Omega)}}{\norm{u}_{L^q(\Omega)}}\,,\,\,\,\,\,\,\,\,\,\int_{\Omega}|u|^{q-2}u\,dx=0.
 $$
\end{thm}

\begin{proof}
Let $u=u_1\chi_{B_1}-u_2\chi_{B_2}$ and $\hat{u}=\hat{u_1}\chi_{B_1}-\hat{u_2}\chi_{B_2}$
be two functions at which $\lambda^{p,q}(\Omega)$ is attained.
We can assume that $u_i$ and $\hat{u_i}$, for $i=1,2$, are radial by Lemma \ref{Damascelli-Sciunzi} in the appendix.
Moreover $u_1, u_2, \hat{u_1}, \hat{u_2}$ are nonnegative and
\begin{equation}\label{contrainte_moyenne}
\int_{B_1}u_1^{q-1}\,dx=\int_{B_2}u_2^{q-1}\,dx\,,\,\,\,\,\,\,\,\,\,\,\,\int_{B_1}\hat{u_1}^{q-1}\,dx=\int_{B_2}\hat{u_2}^{q-1}\,dx\,.
\end{equation}
Without loss of generality we can assume that
  $\norm{u}_{L^q(\Omega)}=\norm{\hat{u}}_{L^q(\Omega)}=[\lambda^{p,q}(\Omega)]^{\frac{p}{q-p}}$. Therefore
\begin{equation}\label{contrainte_norme_q}
\int_{B_1}u_1^{q}\,dx+\int_{B_2}u_2^{q}\,dx=\int_{B_1}\hat{u_1}^{q}\,dx+\int_{B_2}\hat{u_2}^{q}\,dx\,.
\end{equation}
We remark that, by Theorem \ref{idee_daco_gangbo_subia}, letting $r=|x|$, we have that $u_1=u_1(r)$ and $\hat{u_1}=\hat{u_1}(r)$
satisfy on $[0,R_1]$, the Cauchy problem
\begin{equation*}
\left\{
\begin{array}{l}
-(r^{N-1}|\phi'|^{p-2}\phi')'= r^{N-1}|\phi|^{q-1} \vspace{0.1cm}
\\
\phi(0)=c\,,\,\,\,\,\,\phi'(0)=0
\end{array}
\right.
\end{equation*}
with possibly different constants $c$ for $u_1(0)$ and $\hat{u_1}(0)$,
and a similar result holds for $u_2=u_2(r)$ and $\hat{u_2}=\hat{u_2}(r)$ on $[0,R_2]$.

Assume, without loss of generality, that $u_1(0)>\hat{u_1}(0)$. By Lemma \ref{franzina_lamberti} in the appendix, $u_1(r)\geq \hat{u_1}(r)$ on $[0,R_1]$. Therefore
$$
\int_{B_1}u_1^{q-1}\,dx> \int_{B_1}\hat{u_1}^{q-1}\,dx\,,\,\,\,\,\,\,\,\,\,\int_{B_1}u_1^{q}\,dx> \int_{B_1}\hat{u_1}^{q}\,dx\,.
$$
By (\ref{contrainte_moyenne})
\begin{equation}\label{suite_contrainte_moyenne}
\int_{B_2}u_2^{q-1}\,dx> \int_{B_2}\hat{u_2}^{q-1}\,dx\,;
\end{equation}
using (\ref{contrainte_norme_q}) we deduce that
\begin{equation}\label{suite_contrainte_norme_q}
\int_{B_2}u_2^{q}\,dx< \int_{B_2}\hat{u_2}^{q}\,dx\,.
\end{equation}
If $u_2(0)\geq \hat{u_2}(0)$, by Lemma \ref{franzina_lamberti}, we have  $u_2(r)\geq \hat{u_2}(r)$ on $[0,R_2]$ and this is in contradiction with
(\ref{suite_contrainte_norme_q}). If, on the other hand, $u_2(0)< \hat{u_2}(0)$, again using Lemma \ref{franzina_lamberti} we have  $u_2(r)\leq \hat{u_2}(r)$ on $[0,R_2]$ and this contradicts (\ref{suite_contrainte_moyenne}). Then $u_1(0)=\hat{u_1}(0)$. Finally, another application of Lemma \ref{franzina_lamberti} gives us that $u_1=\hat{u_1}$ and
by (\ref{contrainte_norme_q}) we get also $u_2=\hat{u_2}$.
\end{proof}

Now, by \eqref{BVP-limit} (cf. \cite{Emamizadeh:2008p295,Brock:2002p294}) we have that $\dot{u_{1}}$ and $\dot{u_{2}}$ solve, in $B_{1}$ and $B_{2}$ respectively, the equation
\begin{equation}\label{eq-der}
-{\rm div}\left( (p-2)|\nabla u|^{p-1} \frac{\nabla u \cdot \nabla\dot{u}}{|\nabla u|^{3}} \nabla u + |\nabla u|^{p-2} \nabla \dot{u}\right)= p \lambda^{p-1} \dot{\lambda} u^{q-1}+(q-1)\lambda^{p}u^{q-2}\dot{u}
\end{equation}
where we use the notation $\lambda$ instead of $\lambda^{p,q}(B_1\cup B_2)$.
We are in position to prove the optimality condition for the radial problem associated to $\lambda^{p,q}(B_1\cup B_2)$.

\begin{thm}\label{Antoine}
Consider the following minimization problem
\begin{equation}\label{minimiz}
\inf\{\lambda^{p,q}(B_1\cup B_2)\;\;:\;\; (B_1,B_2) \textrm{ disjoint balls},\,\, |B_1\cup B_2|=\omega_N\}\,.
\end{equation}
Let  the pair $(\tilde{B}_{1},\tilde{B}_{2})$ be critical for \eqref{minimiz}. Then, denoted by $u=u_1\chi_{\tilde{B_1}}-u_2\chi_{\tilde{B_2}}$ the function at which $\lambda^{p,q}(\tilde{B}_{1}\cup \tilde{B}_{2})$ is attained, we have
\[
\left|\frac{\partial u_1}{\partial \nu_1}\right|=\left|\frac{\partial u_2}{\partial \nu_2}\right|\,.
\]
\end{thm}
\begin{proof}
 We will denote $B_1\cup B_2$ by $\Omega$. We recall that $u_{1}$ and $u_{2}$ satisfy \eqref{eq-der}. Multiplying by $u$ and integrating we obtain
\begin{equation}\label{prima}
(p-1)\int_{\Omega}|\nabla u|^{p-2}\nabla u\cdot \nabla \dot{u}\,dx=p \lambda^{p-1} \dot{\lambda}\int_{\Omega}|u|^{q}\,dx+(q-1)\lambda^{p}\int_{\Omega}|u|^{q-2}u\dot{u}\,dx.
\end{equation}
Since we are working with normalized functions, the domain derivative of $\|u\|_{L^{q}(\Omega)}^{q}$ has to be zero, i.e. by \eqref{vol-constr}, using \eqref{Functional-Derived}, we deduce that
$$
\int_{\Omega}|u|^{q-2}u\dot{u}\,dx=-\int_{\partial \Omega} |u|^{q}\, V\cdot \nu \, d\mathcal{H}^{N-1}
\,.
$$
As $u$ vanishes on $\partial \Omega$,
\begin{equation}\label{der-cos-2}
\int_{\Omega}|u|^{q-2}u\dot{u}\,dx=0\,.
\end{equation}
%
%
%
%
%
Moreover, by Theorem \ref{idee_daco_gangbo_subia},
$u$ satisfies
$$
-{\rm div}(|\nabla u|^{p-2}\nabla u)=\lambda^p  |u|^{q-2}u\,;
$$
this implies that
\begin{equation}\label{get-it}
\displaystyle \int_{\Omega}|\nabla u|^{p-2}\nabla u\cdot \nabla \dot{u}\,dx=\int_{\partial \Omega}\dot{u}\left|\frac{\partial u}{\partial \nu}\right|^{p-2}\frac{\partial u}{\partial \nu} \, d\mathcal{H}^{N-1}+
\lambda^p\int_{\Omega}|u|^{q-2}u\dot{u}\,dx.
\end{equation}
Combining \eqref{der-cos-2} and \eqref{get-it},  \eqref{prima} reduces to
\[
\dot{\lambda} \frac{p}{p-1} \lambda^{p-1}= \int_{\partial \Omega}\dot{u}\left|\frac{\partial u}{\partial \nu}\right|^{p-2}\frac{\partial u}{\partial \nu} \, d\mathcal{H}^{N-1}\,.
\]
We recall that on $\partial \Omega$, by \eqref{BVP-limit}, we have
\[
\dot{u}=-\frac{\partial u}{\partial \nu} V\cdot \nu;
\]
as a consequence
\[
\dot{\lambda} \frac{p}{p-1} \lambda^{p-1}=- \int_{\partial \Omega}\left|\frac{\partial u}{\partial \nu}\right|^{p}\, V\cdot \nu  \, d\mathcal{H}^{N-1}.
\]
It follows that
\[
\dot{\lambda}^{p,q}(\Omega)=0 \; \Longleftrightarrow \; \int_{\partial \Omega}V\cdot \nu\left|\frac{\partial u}{\partial \nu}\right|^{p}\, d\mathcal{H}^{N-1}=0.
\]
Since $\Omega=B_1\cup B_2$  and since $u$ is radial on $B_1$ and on $B_2$,
this is equivalent to
$$
\left|\frac{\partial u_1}{\partial \nu_1}\right|^{p}\int_{\partial B_1}V\cdot \nu\, d\mathcal{H}^{N-1}+
\left|\frac{\partial u_2}{\partial \nu_2}\right|^{p}\int_{\partial B_2}V\cdot \nu\, d\mathcal{H}^{N-1}
=0\,.
$$
For variations $V$ preserving the volume, we must choose $V$ such that   $\displaystyle \int_{\Omega}{\rm div}(V)\,dx=0$. We deduce
$$\dot{\lambda}^{p,q}(\Omega)=0 \Longleftrightarrow \left|\frac{\partial u_1}{\partial \nu_1}\right|^p-
\left|\frac{\partial u_2}{\partial \nu_2}\right|^p=0\,,
$$
that implies the claim.
\end{proof}

Using the previous analysis and the  Pohozaev type identity (\ref{pohozaev_identity}) in the appendix, we will uniquely identify the critical domain for $\lambda^{p,q}(\Omega)$.

\begin{thm}\label{dominio_critico}
The only critical domain among union of balls of given volume for $\lambda^{p,q}(\Omega)$ is the union of two balls of same measure.
\end{thm}
\begin{proof}
Let $u=u_1\chi_{{B_1}}-u_2\chi_{{B_2}}$ be the function at which $\lambda^{p,q}({B}_{1}\cup {B}_{2})$ is realized, satisfying \eqref{eq-der}.
By the divergence theorem applied to (\ref{equazione_eulero_dgs}) one has
$$
\left|\frac{\partial u_1}{\partial \nu_1}\right|^{p-1}|\partial B_1|=
\left|\frac{\partial u_2}{\partial \nu_2}\right|^{p-1}|\partial B_2|\,.
$$
By the above equality and Theorem \ref{Antoine} we easily deduce that
\begin{equation}\label{derniere}
\left|\frac{\partial u_1}{\partial \nu_1}\right|^{p-2}\frac{\partial u_1}{\partial \nu_1}|\partial B_1|=
\left|\frac{\partial u_1}{\partial \nu_1}\right|^{p-2}\frac{\partial u_1}{\partial \nu_1}|\partial B_2|\,.
\end{equation}
We are going to show, arguing by contradiction, that  $\frac{\partial u_1}{\partial \nu_1}\neq 0$ on  $\partial \Omega$. Indeed, if this is not true, thanks to the regularity given by Lemma \ref{Damascelli-Sciunzi}, we can use Theorem \ref{pohozaev_identity} to infer that
$$
[\lambda^{p,q}(\Omega)]^p\left(\frac{N-p}{p}-\frac Nq\right)=-\frac{p-1}{p}\left|\frac{\partial u_1}{\partial \nu_1}\right|^p\left[\int_{\partial B_1}(x\cdot \nu)d\mathcal{H}^{N-1}+
\int_{\partial B_2}(x\cdot \nu)d\mathcal{H}^{N-1}\right]
= 0
$$
which gives in turns that $\frac{N-p}{p}-\frac Nq = 0$. The last equality contradicts  hypotheses (\ref{hypotheses_pqN}) on $p,q, N$.

Therefore (\ref{derniere}) is equivalent to
$|\partial B_1|=|\partial B_2|$ and so the two balls $B_1$ and $B_2$ have the same radius.
\end{proof}

We are now in position to prove Theorem \ref{main_theorem}.
\begin{proof}[Proof of Theorem  \ref{main_theorem}]
We recall that by Theorem \ref{thmpremiereinegalite} we have
$$\lambda^{p,q}(\Omega)\geq \lambda^{p,q}(B_+\cup B_-)\,.$$
Moreover, using  Theorem \ref{dominio_critico}, we infer that
$$\lambda^{p,q}(B_+\cup B_-)\geq \lambda^{p,q}(B_{1}'\cup B_{2}')$$ where $B_{1}'$ and $B'_{2}$ are disjoint balls such that
 $|B_{1}'|=|B_{2}'|= \frac{|B_+ \cup B_-|}{2}$.

We now observe that, if, $\Omega_1$ and $\Omega_2$ are two sets, with $\Omega_1\subseteq \Omega_2$, then
$\lambda^{p,q}(\Omega_1)\geq \lambda^{p,q}(\Omega_2)$. Indeed, it suffices to consider a function $u_1$ in which $\lambda^{p,q}(\Omega_1)$ is attained
and defining  $0$ in $\Omega_2\setminus \Omega_1$. Therefore
$$
\lambda^{p,q}(B_{1}'\cup B_{2}')\geq \lambda^{p,q}(B_{1}\cup B_{2})\,,
$$
where $B_{1}$ and $B_{2}$ are disjoint balls of measure $\frac{|\Omega|}{2}$.
Combining the previous inequalities we end up with
$$\lambda^{p,q}(\Omega)\geq \lambda^{p,q}(B_1\cup B_2)\,.$$
This proves the claim.
\end{proof}
\begin{rem}
One could ask about the limit as $p\to 1$ of $\lambda^{p,q}(\Omega)$.
Observe that the limit, as $p\to 1$, of
$$
\alpha_{p}(\Omega)=
\inf
\left\{
\frac{\norm{\nabla v}_{L^p(\Omega)}}{\norm{v}_{L^p(\Omega)}}, v\neq 0, v\in W^{1,p}_{0}(\Omega)
\right\}
$$
is the Cheeger constant, defined by
$$
\inf_{D\subset \Omega}\frac{\mathcal{H}^{N-1}(\partial D)}{|D|}
$$
with $D$ varying on all smooth subdomains of $\Omega$ whose boundary does not touch $\partial \Omega$,
as proved by Kawohl and Fridman \cite{kawohl-fridman}.
The limit of $\lambda^{p,q}(\Omega)$ as $p\to 1$ seems to be much more difficult, due to the presence of the parameter $q$ not necessarily equal to $p$ and the non-local constraint
$\displaystyle \int_{\Omega}|u|^{q-2}q=0$.
\end{rem}

\section{Appendix}
We recall here some results about quasilinear elliptic equations. The first (see \cite{damascelli_sciunzi} and the references therein)
gives the radial simmetry of positive solutions to $p$-laplacian equations.
\begin{thm}\label{Damascelli-Sciunzi}
Let $u\in W^{1,p}_0(B)$  be a positive solution to
$-\textnormal{div}(|\nabla u|^{p-2}\nabla u)=\lambda |u|^{q-2}u$, where $B$ is a ball.
Then $u\in C^{1,\alpha}(\overline{B})$ for some $\alpha>0$ and $u$ is radial.
\end{thm}
The next result is a useful comparison lemma for solutions of an initial value problem for the ordinary differential equation arising when one writes in radial coordinates the Euler-Lagrange equation of $\lambda^{p,q}(\Omega)$. This result is widely discussed for example in \cite{Franchi:1996p241} and \cite{MR1462272} and used in this form in \cite{MR2602859}.
\begin{lemma}\label{franzina_lamberti}
Under hypotheses (\ref{hypotheses_pqN}) and $N>1$,
the  Cauchy problem
\begin{equation}\label{pb_cauchy}
\left\{
\begin{array}{l}
-(r^{N-1}|\phi'|^{p-2}\phi')'= r^{N-1}|\phi|^{q-1}
\\
\phi(0)=c\,,\,\,\,\,\,\phi'(0)=0
\end{array}
\right.
\end{equation}
has at most a positive solution on $[0,R]$ of class $C^1([0,R])\cap C^2((0,R))$.
Moreover, let $\phi_1, \phi_2$ be two positive solutions
with $c=c_1$ and $c_2$ respectively;  if $c_1< c_2$, then $\phi_1\leq \phi_2$ on $[0,R]$.
\end{lemma}
The proof of the previous lemma goes exactly as the  one of the Lemmata 3.1 and 3.3 of \cite{MR1462272}. The only comment to make concers the slightly restrictive hypotheses on the values of $p$ and $q$ that we find in \cite{MR1462272}. We observe that we need to check that setting $u(0)=\alpha$, we can invert the unique solution of \eqref{pb_cauchy} $u=u(t,\alpha)$ and this is ensured by
Propositions 1.2.6
and  A2
 in \cite{Franchi:1996p241}. Once we have this, the only hypothesis that has to be satisfied for the applicability of the results in \cite{MR1462272} is the following inequality
\[
\begin{array}{cc}
[(N-p)|s|^{q-2}s  +(N-p)s(q-1)|s|^{q-2}-Np |s|^{q-2}s]|s|^{q-2}s
\vspace{0.2cm}
\\
 \leq (q-1)|s|^{q-2}[(N-p)|s|^q-\frac{Np}{q}|s|^q]\,,
\end{array}
\]
that is,
\begin{equation}\label{ultima}
(N-p)q\leq Np\,.
\end{equation}
If $N-p\leq 0$, clearly we have \eqref{ultima} for any $q$. If on the contrary we have $N-p>0$, then \eqref{ultima} is satisfied exatly for $q\leq p^*$ as in our hypotheses.

We finally recall the following generalization of the Pohozaev identity,
established in \cite{Pucci:1986p447} and \cite{dinca_isaia}.
\begin{thm}
Let $\displaystyle G(u)=\int_0^u g(s)\,ds$ where $g:\R \to \R$ is a continuous function.
Let $\Omega\subset \R^N$, $N\geq 2$, be an open bounded set of class $C^1$.
Let $u \in W^{2,p}({\Omega})\cap W^{1,p}_0(\Omega)$ be a solution to
$$
\left\{
\begin{array}{ll}
-{\rm div}(|\nabla u|^{p-2}\nabla u)=g(u)& \textrm{in}\,\,\Omega
\\
u=0& \textrm{on}\,\,\partial\Omega\,.
\end{array}
\right.
$$
Then
\begin{equation}\label{pohozaev_identity}
-\frac{p-1}{p}\int_{\partial \Omega}|\nabla u|^p(x\cdot \nu)d\mathcal{H}^{N-1}=
\int_{\Omega}\left[
\frac{N-p}{p}|\nabla u|^p -NG(u)
\right]\,dx\,.
\end{equation}
\end{thm}

\section*{Acknowledgements}
The authors would like to thank B. Kawohl for several suggestions.
Part of this work was done during visits of G. Croce to
Institut Elie Cartan Nancy and of G. Pisante to the School of Science of the Hokkaido University whose hospitality is gratefully
acknowledged. 
The work of A. Henrot and G. Croce is part of the project ANR-09-BLAN-0037 {\it Geometric analysis of optimal shapes (GAOS)} financed by the French Agence Nationale de la Recherche (ANR). The work of G. Pisante was partially supported by the Marie Curie project IRSES-2009-247486 of the Seventh Framework Programme.

\bibliographystyle{plain}

\end{document}